\documentclass{amsart}[12pt]
\usepackage{xypic}
\usepackage{eucal}

\usepackage{amsmath}
\usepackage{amstext}
\usepackage{amssymb}
\usepackage{amsthm}
\usepackage{amscd}

\usepackage{lscape}
\usepackage{longtable}

\usepackage{epsfig}
\input txdtools
\let\et=\etexdraw
\def\etexdraw{\drawbb\et}

\theoremstyle{plain}
\newtheorem{thm}{Theorem}[section]
\newtheorem{thm*}{Theorem}

\newtheorem{prop}[thm]{Proposition}
\newtheorem{prop*}[thm*]{Proposition}

\newtheorem{conj}[thm]{Conjecture}

\theoremstyle{definition}

\theoremstyle{remark}

\DeclareMathOperator{\Hom}{Hom}

\begin{document}

\title
[A non-finitely generated algebra of Frobenius maps]
{A non-finitely generated algebra of Frobenius maps}

\author{Mordechai Katzman}
\address{Department of Pure Mathematics,
University of Sheffield, Hicks Building, Sheffield S3 7RH, United Kingdom}
\email{M.Katzman@sheffield.ac.uk}

\subjclass{Primary 13A35, 13E10}


\maketitle

\section{Introduction}\label{Section: Introduction}

The purpose of this paper is to answer a question raised by Gennady Lyubeznik and Karen Smith in \cite{Lyubeznik-Smith}.
This question involves the finite generation of a certain non-commutative algebra which we define below (cf.~section 3 in \cite{Lyubeznik-Smith}.)

Let $S$ be any commutative algebra of prime characteristic $p$.
For any $S$-module $M$ and all $e\geq 0$ we let $\mathcal{F}^e(M)$ denote the
set of all additive functions $\phi: M \rightarrow M$ with the property that $\phi(s m)=s^{p^e} \phi(m)$ for all
$s\in S$ and $m\in M$.
Note that for all $e_1, e_2 \geq 0$, and $\phi_1\in  \mathcal{F}^{e_1}(M)$, $\phi_2\in  \mathcal{F}^{e_2}(M)$
the composition
$\phi_2 \circ \phi_1$ is in $\mathcal{F}^{e_1+e_2}(M)$.
Note also that each $\mathcal{F}^{e}(M)$ is a module over $\mathcal{F}^{0}(M)=\Hom_{S}(M,M)$ via
$\phi_0 \phi=\phi_0 \circ \phi$.
We now define $\mathcal{F}(M)=\oplus_{e\geq 0}  \mathcal{F}^e(M)$ and endow it with the structure of a
$\Hom_{S}(M,M)$-algebra with multiplication given by composition.


In section \ref{Section: The example} below we construct an example of an Artinian module over a complete local ring $S$ for which
$\mathcal{F}(M)$ is not a finitely generated $\Hom_{S}(M,M)$-algebra, thus giving a negative answer to the question raised in section 3 of
\cite{Lyubeznik-Smith}.

\section{The example}\label{Section: The example}

Let $\mathbb{K}$ be a field of characteristic $p>0$, $R=\mathbb{K}[\![ x,y,z ]\!]$, and let $I\subseteq R$ be an ideal.
Let $E$ be the injective hull of the residue field of $R$ and let $f$ denote the standard Frobenius map of $E$ (cf.~section 4 in \cite{Katzman}.)
Write $S=R/I$ and let $E_S$ be the injective hull of the residue field of $S$.

Notice that as $S$ is complete, $\mathcal{F}^0(E_S)=\Hom_S(E_S,E_S)\cong S$;
the $S$-module $\mathcal{F}^e(E_S)$ of $p^e$th Frobenius maps on $E_S$ is given by
$(I^{[p^e]} : I) f^e$ (cf.~section 4 in \cite{Katzman}.)

For all $e\geq 1$ write $K_e=(I^{[p^e]} : I)$.
We define
$$L_e = \sum_{1\leq \beta_1, \dots, \beta_s < e\atop \beta_1 + \dots +  \beta_s =e}
K_{\beta_1} K_{\beta_2}^{[p^{\beta_1}]} K_{\beta_3}^{[p^{\beta_1+\beta_2}]} \cdot \dots \cdot K_{\beta_s}^{[p^{\beta_1+\dots +\beta_{s-1}}]} .$$

\begin{prop}
Fix any $e\geq 1$, and let $\mathcal{F}_{<e}$ be the $S$-subalgebra of $\mathcal{F}(E_S)$ generated by
$\mathcal{F}^0(E_S), \dots, \mathcal{F}^{e-1}(E_S)$.
We have $\mathcal{F}_{<e} \cap \mathcal{F}^e(E_S)=L_e f^e$.
\end{prop}
\begin{proof}
Any element in $\mathcal{F}_{<e} \cap \mathcal{F}^e(E_S)$ can be written as a sum of elements
of the form
$\phi_1 \cdot \dots \cdot \phi_s$ where for all $1\leq j\leq s$ we have
$\phi_j\in   \mathcal{F}^{\beta_j}(E_S)$ ($1\leq \beta_j<e$)
and $\beta_1+ \dots + \beta_s=e$.
Each such $\phi_j$ equals $a_j f^{\beta_j}$ where $a_j\in K_{\beta_j}$, so
$$
\phi_1 \cdot \dots \cdot \phi_s =
a_1 f^{\beta_1} a_2 f^{\beta_2} a_3 f^{\beta_3} \cdot \dots \cdot a_s f^{\beta_s}=
a_1  a_2^{p^\beta_1} a_3^{p^{\beta_1+\beta_2}} \cdot \dots \cdot a_s^{p^{\beta_1 + \dots + \beta_{s-1}}} f^{\beta_1 + \dots + \beta_s} \in L_e f^e$$
so $\mathcal{F}_{<e} \cap \mathcal{F}^e(E_S) \subseteq L_e f^e$.

On the other hand, for all $1\leq \beta_1, \dots, \beta_s < e$ such that $\beta_1 + \dots +  \beta_s=e$,
$$K_{\beta_1} K_{\beta_2}^{[p^{\beta_1}]} K_{\beta_3}^{[p^{\beta_1+\beta_2}]} \cdot \dots \cdot K_{\beta_s}^{[p^{\beta_1+\dots +\beta_{s-1}}]}
\subseteq (I^{[p^{\beta_1+\dots+\beta_s}]} : I) = (I^{[p^e]} : I)$$
so $L_e f^e \subseteq (I^{[p^e]} : I) f^e= \mathcal{F}^e(E_S)$. A similar argument to the one in the previous paragraph shows that we also have
$$K_{\beta_1} K_{\beta_2}^{[p^{\beta_1}]} K_{\beta_3}^{[p^{\beta_1+\beta_2}]} \cdot \dots \cdot K_{\beta_s}^{[p^{\beta_1+\dots +\beta_{s-1}}]} f^e
\subseteq \mathcal{F}_{<e}$$
and we deduce that $L_e f^e \subseteq \mathcal{F}_{<e} \cap \mathcal{F}^e(E_S)$.
\end{proof}

Fix now $I$ to be the ideal generated by $xy$ and $yz$.
We show that $\mathcal{F}(M)$ is not a finitely generated $S$-algebra.

\begin{prop}
For all $e\geq 1$, $K_e$ is generated by
$$\left\{ x^{p^e} y^{p^e-1}, x^{p^e-1} y^{p^e-1} z^{p^e-1}, y^{p^e-1}z^{p^e} \right\} .$$
\end{prop}
\begin{proof}
For any $q>1$,
\begin{eqnarray*}
( x^{q} y^{q} , y^{q} z^{q} ) : (x y , y z )& = & \left( ( x^{q} y^{q}, y^{q} z^{q} ) : x y  \right) \cap
                                                            \left( ( x^{q} y^{q}, y^{q} z^{q} ) :  y z  \right)  \\
    & = & (x^{q-1} y^{q-1} , y^{q-1} z^{q}) \cap   (x^{q} y^{q-1}, y^{q-1} z^{q-1}) \\
    & = & (x^{q} y^{q-1} , x^{q-1} y^{q-1} z^{q-1} , x^{q} y^{q-1}z^{q} , y^{q-1} z^{q} ) \\
    & = & (x^{q} y^{q-1} , x^{q-1} y^{q-1} z^{q-1} ,  y^{q-1} z^{q} )
\end{eqnarray*}
\end{proof}

\begin{thm}
The $S$-algebra $\mathcal{F}(E_S)$ is not finitely generated.
\end{thm}
\begin{proof}
It is enough to show that for all $e\geq 1$, $\mathcal{F}(E_S)$ is not in $\mathcal{F}_{<e}$ and we establish this by showing that the
generator $x^{p^e} y^{p^e-1}$ of $K_e$ is not in $L_e$.

Since $L_e$ is a sum of monomial ideals, $x^{p^e} y^{p^e-1}\in L_e$ if and only if $x^{p^e} y^{p^e-1}$ is in one of the summands.
So we now fix $e\geq 1$ and
$1\leq \beta_1,  \dots , \beta_s<e$ such that $\beta_1 + \dots +  \beta_s =e$,
and show that the ideal
$$K_{\beta_1} K_{\beta_2}^{[p^{\beta_1}]} K_{\beta_3}^{[p^{\beta_1+\beta_2}]} \cdot \dots \cdot K_{\beta_s}^{[p^{\beta_1+\dots +\beta_{s-1}}]} $$
does not contain
$x^{p^e} y^{p^e-1}$.

Since $z$ does not occur in $x^{p^e} y^{p^e-1}$, it is enough to show that with $J_e=x^{p^e} y^{p^e-1} R$,
$$J_{\beta_1} J_{\beta_2}^{[p^{\beta_1}]} J_{\beta_3}^{[p^{\beta_1+\beta_2}]} \cdot \dots \cdot J_{\beta_s}^{[p^{\beta_1+\dots +\beta_{s-1}}]} $$
does not contain
$x^{p^e} y^{p^e-1}$.
The exponent of $x$ in the generator of the product above is
$$p^{\beta_1 + (\beta_1+\beta_2) + \dots + (\beta_1+\dots +\beta_s)}>p^{\beta_1+\dots +\beta_s}=p^e $$
where the inequality follows from the fact that we must have $s>1$.
\end{proof}

\section{A conjecture}\label{Section: A conjecture}

Although the example in section \ref{Section: The example} settles the question raised in \cite{Lyubeznik-Smith}, one
might still raise the question of whether such examples exist over ``nice'' rings, e.g., normal domains.

Let $\mathbb{K}$ be a field of prime characteristic $p$, let $R=\mathbb{K}[\![ x,y,z,u,v,w]\!]$ and let $I$ be the ideal
generated by the $2\times 2$ minors of the matrix
$\left(
\begin{array}{lll} x&y&z\\ u&v&w \end{array}
\right)$.

The ring $S=R/I$ is a normal, Cohen-Macaulay domain (cf.~Theorem 7.3.1 in \cite{Bruns-Herzog}.)
Let $E_S$ be the injective hull of the residue field of $S$ and, as before, for all $e\geq 1$
let $\mathcal{F}_{<e}$ be the $S$-subalgebra of $\mathcal{F}^e(E_S)$ generated by
$\mathcal{F}^1(E_S), \dots , \mathcal{F}^{e-1}(E_S)$. Note that $\mathcal{F}^0(E_S)=S$.

\begin{conj}
For all $e\geq 1$, $\mathcal{F}^e(E_S)$ is not contained in $\mathcal{F}_{<e}$ and hence
$\mathcal{F}^e(E_S)$ is not a finitely generated $S$-algebra.
\end{conj}

I have tested this conjecture using the computer system Macaulay2 (\cite{Macaulay2}), and, for example, in characteristic 2, it holds
for $1\leq e\leq 6$.

\end{document}